\documentclass[draft]{amsart}
\usepackage{amsmath}
\usepackage{amssymb}

\newtheorem{thm}{Theorem}
\newtheorem{lem}{Lemma}

\theoremstyle{remark}
\newtheorem{rmk}{Remark}

\newcommand{\cv}{\mathbf{C}}

\newcommand{\ball}{\mathbb{B}}
\def\Re{{\sf Re}}

\begin{document}

\title[Estimate of the squeezing function]{Estimate of the squeezing function for a class of bounded domains}

\author[J.E. Forn\ae ss, F. Rong]{John Erik Forn\ae ss, Feng Rong}

\address{Department of Mathematics, NTNU, Sentralbygg 2, Alfred Getz vei 1, 7491 Trondheim, Norway}
\email{johnefo@math.ntnu.no}

\address{Department of Mathematics, School of Mathematical Sciences, Shanghai Jiao Tong University, 800 Dong Chuan Road, Shanghai, 200240, P.R. China}
\email{frong@sjtu.edu.cn}



\thanks{The first named author is partially supported by grant DMS1006294 from the National Science Foundation and grant 240569 from the Norwegian Research Council. The second named author is partially supported by grant 11371246 from the National Natural Science Foundation of China.}

\begin{abstract}
We construct a class of bounded domains, on which the squeezing function is not uniformly bounded from below near a smooth and pseudoconvex boundary point.
\end{abstract}

\maketitle


\section{Introduction}

In \cite{LSY1, LSY2}, the authors introduced the notion of \textit{holomorphic homogeneous regular}. Then in \cite{Y:squeezing}, the equivalent notion of \textit{uniformly squeezing} was introduced. Motivated by these studies, in \cite{DGZ1}, the authors introduced the \textit{squeezing function} as follows.

Denote by $\ball(r)$ the ball of radius $r>0$ centered at the origin 0. Let $\Omega$ be a bounded domain in $\cv^n$, and $p\in \Omega$. For any holomorphic embedding $f:\Omega\rightarrow \ball(1)$, with $f(p)=0$, set
$$s_{\Omega,f}(p):=\sup\{r>0:\ \ball(r)\subset f(\Omega)\}.$$
Then, the squeezing function of $\Omega$ at $p$ is defined as
$$s_\Omega(p):=\sup_f\{s_{\Omega,f}(p)\}.$$

Many properties and applications of the squeezing function have been explored by various authors, see e.g. \cite{DGZ1,DGZ2,DF:Bergman,DFW:squeezing,FS:squeezing,FW:squeezing,KZ:squeezing}.

It is clear that squeezing functions are invariant under biholomorphisms, and they are positive and bounded above by 1. It is a natural and interesting problem to study the uniform lower and upper bounds of the squeezing function.

It was shown recently in \cite{KZ:squeezing} that the squeezing function is uniformly bounded below for bounded convex domains (cf. \cite[Theorem 1.1]{Frankel}). On the other hand, in \cite{DGZ1}, the authors showed that the squeezing function is not uniformly bounded below on certain domains with non-smooth boundaries, such as punctured balls. In \cite{DF}, the authors constructed a smooth pseudoconvex domain in $\cv^3$ on which the quotient of the Bergman metric and the Kobayashi metric is not bounded above near an infinite type point. By \cite[Theorem 3.3]{DGZ2}, the squeezing function is not uniformly bounded below on this domain.

These studies raise the question: Is the squeezing function always uniformly bounded below near a smooth finite type point? In this paper, we answer the question negatively. More precisely, we have the following

\begin{thm}\label{T:main}
Let $\Omega$ be a bounded domain in $\cv^3$, and $q\in \partial\Omega$. Assume that $\Omega$ is smooth and pseudoconvex in a neighborhood of $q$ and the Bloom-Graham type of $\Omega$ at $q$ is $d<\infty$. Moreover, assume that the regular order of contact at $q$ is greater than $2d$ along two smooth complex curves not tangent to each other. Then the squeezing function $s_\Omega(p)$ has no uniform lower bound near $q$.
\end{thm}

\begin{rmk}
The proof gives the estimate $s_\Omega(p)\leq C \delta^{\frac{1}{2d(2d+1)}}$ for some points approaching the boundary.
\end{rmk}

In section \ref{S:prelim}, we recall some preliminary notions and results. In section \ref{S:proof}, we prove Theorem \ref{T:main}.

\section{Preliminaries}\label{S:prelim}

Let $\Omega$ be a bounded domain in $\cv^n$, $n\ge 2$, and $q\in \partial\Omega$. Assume that $\Omega$ is smooth and pseudoconvex in a neighborhood of $q$. The \textit{Bloom-Graham type} of $\Omega$ at $q$ is the maximal order of contact of complex manifolds of dimension $n-1$ tangent to $\partial \Omega$ at $q$ (see e.g. \cite{BG}). Choose local coordinates $(z,t)\in \cv^{n-1}\times \cv$ such that the complex manifold of dimension $n-1$ with the maximal order of contact is given by $\{t=0\}$. Then $\Omega$ is locally given by $\rho(z,t)<0$, where $\rho(z,t)=\Re t+P(z)+Q(z,t)$ with $Q(z,0)\equiv 0$ and $\deg P(z)=d$. (We say that the degree of $P$ is $d$ if the Taylor expansion of $P$ has no nonzero term of degree less than $d$.) Since $\Omega$ is pseudoconvex, we actually have $d=2k$ (see e.g. \cite{D'Angelo}).

For $1\le k\le n-1$, let $\varphi:\cv^k\rightarrow \cv^n$ be analytic with $\varphi(0)=q$ and $\textup{rank} d\varphi(0)=k$. Then the \textit{regular order of contact} at $q$ along the $k$-dimensional complex manifold defined by $\varphi$ is defined as $\deg \rho\circ\varphi$ (see e.g. \cite{D'Angelo}).

Denote by $\Delta$ the unit disc in $\cv$. Let $p\in \Omega$ and $\zeta\in \cv^n$. The \textit{Kobayashi metric} is defined as
$$K_\Omega(p,\zeta):=\inf\{\alpha:\ \alpha>0,\ \exists\ \phi:\Delta\rightarrow \Omega,\ \phi(0)=p,\ \alpha \phi'(0)=\zeta\}.$$
Then the \textit{Kobayashi indicatrix} is defined as (see e.g. \cite{K:indicatrix})
$$D_\Omega(p):=\{\zeta\in \cv^n:\ K_\Omega(p,\zeta)<1\}.$$
For each unit vector $e\in \cv^n$, set $D_\Omega(p,e):=\max\{|\eta|:\ \eta\in \cv,\ \eta e\in D_\Omega(p)\}$. By the definition of Kobayashi indicatrix, the following three lemmas are clear.

\begin{lem}\label{L:ball}
$D_{\ball(r)}(0)=\ball(r)$.
\end{lem}

\begin{lem}\label{L:indicatrix}
Let $\Omega_1$ and $\Omega_2$ be two domains in $\cv^n$ with $\Omega_1\subset \Omega_2$. Then for each $p\in \Omega_1$, $D_{\Omega_1}(p)\subset D_{\Omega_2}(p)$.
\end{lem}

\begin{lem}\label{L:indicatrix1}
Let $\Omega$ be a domain in $\cv^n$ and $f:\Omega\rightarrow \cv^n$ a biholomorphic map. Then for each $p\in \Omega$, $D_{f(\Omega)}(f(p))=f'(p)D_\Omega(p)$.
\end{lem}

We also need the following localization lemma (see e.g. \cite[Lemma 3]{FL:metrics}). We will use $\gtrsim$ (resp. $\lesssim$, $\simeq$) to mean $\ge$ (resp. $\le$, $=$) up to a positive constant.

\begin{lem}\label{L:localization}
Let $\Omega$ be a bounded domain in $\cv^n$, $q\in \partial \Omega$ and $U$ a neighborhood of $q$. If $V\subset\subset U$ and $q\in V$, then
$$K_\Omega(p,\zeta)\simeq K_{\Omega\cap U}(p,\zeta),\ \ \ \forall\ p\in V,\ \zeta\in \cv^n.$$
\end{lem}

By the above lemma, when we consider the size of the Kobayashi indicatrix in the next section, we will work in $\Omega\cap U$.

\section{Estimate of the squeezing function}\label{S:proof}

We first choose local coordinates adapted to our purpose.

\begin{lem}\label{L:normal}
Let $\Omega$ be a bounded domain in $\cv^{n+1}$, $n\ge 1$, and $q\in \partial\Omega$. Assume that $\Omega$ is smooth and pseudoconvex in a neighborhood of $q$ and the Bloom-Graham type of $\Omega$ at $q$ is $2k$, $k\ge 1$. Then there exist local coordinates $(z,t)=(z_1,\cdots,z_n,u+iv)$ such that $q=(0,0)$ and $\Omega$ is locally given by $\rho(z,t)<0$ with
\begin{equation}\label{E:normal}
\rho(z,t)=u+P(z)+Q(z)+vR(z)+u^2+v^2+o(u^2,uv,v^2,u|z|^{2k}),
\end{equation}
where $P(z)$ is plurisubharmonic, homogeneous of degree $2k,$ but not pluriharmonic, $\deg Q(z)\ge 2k+1$ and $\deg R(z)\ge k+1$.
\end{lem}
\begin{proof}
By assumption, we have a local defining function of the form
$$\rho(z,t)=u+P(z)+Q(z)+au^2+buv+cv^2+uA(z)+vB(z)+o(|t|^2),$$
where $P(z)$ is plurisubharmonic, homogeneous of degree $2k,$ but not pluriharmonic,  $\deg Q(z)\ge 2k+1$,
and $\deg A(z),B(z)\ge 1$. By changing $t$ to $t+dt^2$ and multiplying with $1+eu$ or $1+ev$, we can freely change the quadratic terms in $u,v$. Thus, we can assume that
$$\rho(z,t)=u+P(z)+Q(z)+u^2+v^2+uA(z)+vB(z)+o(|t|^2).$$
Multiplying with $1-A(z)$, we get a new $A$, say $A'$ of degree at least $2$. Multiplying with $1-A'(z)$ we get
a new $A$ of degree at least $4$. Continuing, we can further assume that
$$\rho(z,t)=u+P(z)+Q(z)+u^2+v^2+vB(z)+o(|t|^2, u|z|^{2k}).$$

Write $B(z)=B_s(z)+B'(z)$, where $B_s(z)$ is the lowest order homogeneous term of degree $s\ge 1$. Assume that $B_s(z)$ is pluriharmonic. Then there exists a holomorphic function $F(z)=A(z)-iB_s(z)$. Change again $\rho$ to add the term $uA(z)$ with this new $A(z)$. Then $\rho$ takes the form
$$\begin{aligned}
\rho(z,t)&=u+P(z)+Q(z)+u^2+v^2+uA(z)+vB_s(z)+vB'(z)+o(|t|^2,u|z|^{2k})\\
&=u+P(z)+Q(z)+u^2+v^2+\Re(tF(z))+vB'(z)+o(|t|^2,u|z|^{2k}).
\end{aligned}$$
By absorbing $\Re(tF(z))$ into $u$, we get
$$\rho(z,t)=u+P(z)+Q(z)+u^2+v^2+vB'(z)+o(|t|^2,u|z|^{2k}).$$
Continuing this process, we can assume that $\rho$ takes the form
$$\rho(z,t)=u+P(z)+Q(z)+u^2+v^2+vB_l(z)+vB'(z)+o(|t|^2,u|z|^{2k}),$$
where $B_l(z)$ is not pluriharmonic and $l\le k$ or $l\ge k+1$. We assume the first alternative, otherwise the proof is done. We will arrive at a contradiction to pseudoconvexity.

Note that $P(z)$ is plurisubharmonic but not pluriharmonic. This implies that there exists a complex line through the origin on which the restriction of $P$ is subharmonic, but not harmonic (cf. \cite{Forelli}). Pick a tangent vector $\xi=(\xi_1,\dots \xi_n)$ so that the Levi form of $P$ calculated at a point $\eta\xi$, $\eta=|\eta|e^{i\theta}$, in the direction of $\xi$ is $|\eta|^{2k-2}G(\theta)\|\xi\|^2$. Here $G$ is a smooth nonnegative function which at most vanishes at finitely many angles. Choose $\lambda$ such that $\sigma=(\xi,\lambda)$ is a complex tangent vector to $\partial \Omega$, i.e.
$$\sum_{j=1}^n \frac{\partial \rho}{\partial z_j}\xi_j+ \frac{\partial \rho}{\partial t}\lambda=0.$$
Then we have $|\lambda|=O(|\eta|^{2k-1}+|v||\eta|^{l-1}+|t|^2+|u||\eta|^{2k-1})\|\xi\|$.

The Levi form of $\rho$ at a boundary point $(\eta\xi,t)$, along the tangent vector $\sigma$ is
$$\begin{aligned}
\mathcal L(r,\sigma)=&|z|^{2k-2}G(\theta)\|\xi\|^2+|\lambda|^2+\mathcal L(vB_l(z),\sigma)+\cdots\\
=&|z|^{2k-2}G(\theta)\|\xi\|^2+|\lambda|^2+\Re(\sum_{j=1}^n\frac{\partial B_l}{\partial z_j}i\xi_j\overline{\lambda})+v\sum_{k,m}\frac{\partial^2 B_l}{\partial z_k\overline z_m}\xi_k\overline{\xi}_m+\cdots.
\end{aligned}$$

Since $B_l$ is not pluriharmonic, we can assume after changing $\xi$ slightly that $\frac{\partial^2 B_l}{\partial z_k\partial \overline{z}_m}\xi_k\overline{\xi}_m\neq 0$. Next choose $v=\pm C|\eta|^k$ with $C>\max_\theta\{G(\theta)\}$. The second term is $o(|\eta|^{2k-2}\|\xi\|^2)$ and the third terms is $o(|\eta|^{k+l-2}\|\xi\|^2)$. The last term is $\simeq (|\eta|^{k+l-2}\|\xi\|^2)$ and, since $l\le k$, at least $O(|\eta|^{2k-2}\|\xi\|^2)$. Thus we have $\mathcal L(r,\sigma)<0$. This is a contraction.
\end{proof}

By Lemma \ref{L:normal}, we can choose local coordinates $(z,w,t)=(z,w,u+iv)$ near $q$ such that $q=(0,0,0)$ and $\Omega$ is locally given by $\rho(z,w,t)<0$, where
\begin{equation}\label{E:rho}
\rho(z,w,t)=u+P(z,w)+Q(z,w)+vR(z,w)+u^2+v^2+o(u^2,uv,v^2,u|(z,w)|^{2k}).
\end{equation}
Here $P(z,w)$ is homogeneous of degree $2k$ with $P(z,0)=P(0,w)=0$, $\deg Q(z,w)\ge 2k+1$ with $\deg Q(z,0)\ge 4k+1$ and $\deg Q(0,w)\ge 4k+1$, and $\deg R(z,w)\ge k+1$. Set $p=(0,0,-\delta)$ with $0<\delta\ll 1$.

\begin{lem}\label{L:10}
Let $\zeta_1=(1,0,0)$ and $\zeta_2=(0,1,0)$. Then $K_\Omega(p,\zeta_1),K_\Omega(p,\zeta_2)\lesssim \delta^{-\frac{1}{4k+1}}$.
\end{lem}
\begin{proof}
Consider the linear map $\phi:\Delta\rightarrow \cv^3$ with $\phi(\tau)=(\beta\tau,0,-\delta)$ for $\tau\in \Delta$, and $|\beta|=\epsilon \delta^{\frac{1}{4k+1}}$ for $0<\epsilon\ll 1$. Then
$$\rho\circ\phi(\tau)\le -\delta+C|\beta\tau|^{4k+1}+o(\delta)<-\delta+\epsilon\delta+o(\delta)<0.$$
Therefore, $K_\Omega(p,\zeta_1)\lesssim \delta^{-\frac{1}{4k+1}}$. The argument in the direction $\zeta_2$ is similar.
\end{proof}

Let $(a,b,0)$ be a point so that $P(a \tau,b \tau)$ is a subharmonic homogeneous polynomial of degree $2k$ which is not harmonic. Then both $a$ and $b$ must be nonzero. By scaling in each variable, we can assume that $a=b=1/{\sqrt {2}}$. 

\begin{lem}\label{L:11}
Let $\zeta=\frac{1}{\sqrt{2}}(1,1,0).$  Then $K_\Omega(p,\zeta)\gtrsim \delta^{-\frac{1}{4k}}$.
\end{lem}
\begin{proof}
For $z,w$ small, we have
$$\begin{aligned}
v^2/2+vR(z,w)&\ge v^2/2-C|v|\|z,w\|^{k+1}+C^2\|z,w\|^{2k+2}-C^2\|z,w\|^{2k+2}\\
&\ge -C^2\|z,w\|^{2k+2}.
\end{aligned}$$
Therefore,
$$\begin{aligned}
\rho&\ge u+P(z,w)+Q(z,w)-C^2\|z,w\|^{2k+2}+u^2+v^2/2+o(u^2,uv,v^2,u|(z,w)|^{2k})\\
&\ge u+P(z,w)+\tilde{Q}(z,w)+u^2/2+v^2/4+o(u|(z,w)|^{2k})=:\tilde{\rho}.
\end{aligned}$$

Consider an analytic map $\phi:\Delta\rightarrow \Omega$ with
$$\phi(\tau)=(\beta\tau+f(\tau),\beta\tau+g(\tau),-\delta+h(\tau)),\ \ \ |f(\tau)|,|g(\tau)|,|h(\tau)|\le |\tau|^2.$$
Then $\tilde{\rho}\circ\phi(\tau)\leq\rho\circ\phi(\tau)<0$. And we have
$$\begin{aligned}
\tilde{\rho}(\phi(\tau))=&-\delta+\Re h(\tau)+P(\beta\tau+f(\tau),\beta\tau+g(\tau))+\tilde{Q}(\beta\tau+f(\tau),\beta\tau+g(\tau))\\
&+u^2/2+v^2/4+o(u|(z,w)|^{2k})<0,
\end{aligned}$$
and
\begin{equation}\label{E:average}
\frac{1}{2\pi}\int_0^{2\pi} \tilde{\rho}(\phi(|\tau|e^{i\theta}))d\theta<0.
\end{equation}

Note that, by the homogeneous expansion of $P(z,w)$, we have for $|\tau|<1$ small
\begin{equation}\label{E:xitau}
|\beta\tau|^{2k}-\sum_{i=0}^{2k-1} |\beta|^i|\tau|^{4k-i}\lesssim \delta.
\end{equation}
Choose $|\tau|=c|\beta|$ for some small constant $c>0.$  Then \eqref{E:xitau} gives
$$\left|\frac{\beta}{2}\right|^{4k}\lesssim \delta.$$
Hence, $K_\Omega(p,u)\gtrsim \delta^{-\frac{1}{4k}}$. 
\end{proof}

\begin{lem}\label{L:squeezing}
Let $D$ be a bounded domain in $\cv^n$, $n\ge 2$, containing the origin. Assume that there exist two linearly independent nonzero vectors $\zeta_1,\zeta_2\in D$ and $\epsilon>0$ such that $\epsilon(\zeta_1+\zeta_2)\not\in D$. Then there does not exist a linear map $L:D\rightarrow \cv^n$, with $L(0)=0$, such that $\ball(3\epsilon)\subset L(D)\subset \ball(1)$.
\end{lem}
\begin{proof}
Let $L:D\rightarrow \cv^n$ be a linear map with $L(0)=0$ and suppose $\ball(3\epsilon)\subset L(D)$. Since $\epsilon(\zeta_1+\zeta_2)\not\in D$ and $L$ is linear, we have $\epsilon(L(\zeta_1)+L(\zeta_2))\not\in L(D)$. This implies that $\epsilon(L(\zeta_1)+L(\zeta_2))\not\in \ball(3\epsilon)$ and thus $\|L(\zeta_1)+L(\zeta_2)\|\ge 3$. However,
$\|L(\zeta_1)+L(\zeta_2)\|\leq \|L(\zeta_1)\|+\|L(\zeta_2)\|\leq 1+1=2.$
This completes the proof.
\end{proof}

\begin{proof}[Proof of Theorem \ref{T:main}]
Choose local coordinates $(z,w,t)$ such that $q=(0,0,0)$ and let $p=(-\delta,0,0)$ for $\delta>0$ small. Let $\zeta_1=(1,0,0)$ and $\zeta_2=(0,1,0)$, the two directions along which the regular order of contact at $q$ is greater than $2d=4k$. By Lemma \ref{L:10}, $K_\Omega(p,\zeta_1),K_\Omega(p,\zeta_2)\lesssim \delta^{-\frac{1}{4k+1}}$. By Lemma \ref{L:11}, $K_\Omega(p,\frac{1}{\sqrt{2}}(\zeta_1+\zeta_2))\gtrsim \delta^{-\frac{1}{4k}}$.

Choose $\lambda>0$ with $\lambda\gtrsim \delta^{\frac{1}{4k+1}}$ such that $\lambda \zeta_1,\lambda \zeta_2\in D_\Omega(p)$. Then for $\epsilon \simeq \delta^{\frac{1}{4k(4k+1)}}$, we have $\epsilon (\lambda \zeta_1+ \lambda \zeta_2)\not\in D_\Omega(p)$. Thus, by Lemma \ref{L:squeezing}, there does not exist a linear map $L:D_\Omega(p)\rightarrow \cv^3$ such that $\ball(3\epsilon)\subset L(D_\Omega(p))\subset \ball(1)$.

Let $f$ be a biholomorphism of $\Omega$ into $\ball(1)$ such that $f(p)=0$ and $\ball(c)\subset f(\Omega)$ for some $c>0$. Set $L=f'(p)$. Then, by Lemmas \ref{L:ball}, \ref{L:indicatrix} and \ref{L:indicatrix1}, $\ball(c)\subset L(D_\Omega(p))\subset \ball(1)$. Therefore, we have $c\lesssim \delta^{\frac{1}{4k(4k+1)}}$. Since $f$ is arbitrary, we get $s_\Omega(p) \lesssim \delta^{\frac{1}{4k(4k+1)}}$. Since $\delta$ can be arbitrarily small, this completes the proof.
\end{proof}

\begin{rmk}
Theorem \ref{T:main} does not hold if only assuming that the regular order of contact at $q$ is greater than $2d$ along one smooth complex curve. For instance, consider $\Omega$ given by
$$\{(z,w,t)\in \cv^3:\ |t|^2+|z|^2+|w|^6<1\}.$$
Then at $q=(0,0,1)$, the Bloom-Graham type is $2$ and the regular order of contact along $(0,1,0)$ is $6>4$. But $\Omega$ is a bounded convex domain and thus the squeezing function has a uniform lower bound by \cite{KZ:squeezing}.
\end{rmk}

\begin{rmk}
Using similar arguments, one can extend Theorem \ref{T:main} to higher dimensions as follows.

\begin{thm}
Let $\Omega$ be a bounded domain in $\cv^n$, $n\ge 4$, and $q\in \partial\Omega$. Assume that $\Omega$ is smooth and pseudoconvex in a neighborhood of $q$ and the Bloom-Graham type of $\Omega$ at $q$ is $d$. Moreover, assume that the regular order of contact at $q$ is $d$ along a two-dimensional complex surface $\Sigma$ and the regular order of contact at $q$ is greater than $2d$ along two smooth complex curves not tangent to each other contained in $\Sigma$. Then the squeezing function $s_\Omega(p)$ has no uniform lower bound near $q$.
\end{thm}
\end{rmk}

\begin{rmk}
After the completion of this work, it was brought to our attention by Gregor Herbort that a similar comparison result to \cite{DF} was obtained for the following domain in \cite{DFH:Bergman}:
$$\Omega:=\{(z,w,t)\in \cv^3:\ \Re t+|z|^{12}+|w|^{12}+|z|^2|w|^4+|z|^6|w|^2<0\}.$$
Therefore, by our remark in the introduction, the squeezing function does not have a uniform lower bound on this domain. More generally, we have the following
\begin{thm}
Let $\Omega$ be a bounded domain in $\cv^3$, and $q\in \partial\Omega$. Assume that $\Omega$ is smooth and pseudoconvex in a neighborhood of $q$ and the Bloom-Graham type of $\Omega$ at $q$ is $d<\infty$. Let $\rho$ be a defining function of $\Omega$ near $q$ in the normal form \eqref{E:normal} and assume that the leading homogeneous term $P(z)$ only contains positive terms. Moreover, assume that the regular order of contact at $q$ is greater than $d$ along two smooth complex curves not tangent to each other. Then the squeezing function $s_\Omega(p)$ has no uniform lower bound near $q$.
\end{thm}
\begin{proof}[Sketch of proof]
In Lemma \ref{L:10}, we get $K_\Omega(p,\zeta_1),K_\Omega(p,\zeta_2)\lesssim \delta^{-\frac{1}{2k+1}}$, by the same argument. In Lemma \ref{L:11}, we get $K_\Omega(p,\zeta)\gtrsim \delta^{-\frac{1}{2k}}$, by noticing that instead of \eqref{E:xitau} we have $|\xi\tau|^{2k}\lesssim \delta$ since all terms of $P(z)$ are positive. Then arguing exactly as in the proof of Theorem \ref{T:main}, we get $s_\Omega(p) \lesssim \delta^{\frac{1}{2k(2k+1)}}$.
\end{proof}
\end{rmk}


\begin{thebibliography}{}

\bibitem{BG}
T. Bloom, I. Graham;
\emph{A geometric characterization of points of type $m$ on real submanifolds of $\cv^n$},
J. Differential Geom. {\bf 12} (1977), 171-182.

\bibitem{D'Angelo}
J.P. D'Angelo;
Several Complex Variables and the Geometry of Real Hypersurfaces,
Studies in Advanced Mathematics, CRC Press, Boca Raton, FL, 1993.

\bibitem{DGZ1}
F. Deng, Q. Guan, L. Zhang;
\emph{Some properties of squeezing functions on bounded domains},
Pacific J. Math. {\bf 257} (2012), 319-341.

\bibitem{DGZ2}
F. Deng, Q. Guan, L. Zhang;
\emph{Properties of squeezing functions and global transformations of bounded domains},
Trans. Amer. Math. Soc. {\bf 368} (2016), 2679-2696.

\bibitem{DF}
K. Diederich, J.E. Forn\ae ss;
\emph{Comparison of the Bergman and the Kobayashi metric},
Math. Ann. {\bf 254} (1980), 257-262.

\bibitem{DF:Bergman}
K. Diederich, J.E. Forn\ae ss;
\emph{Boundary behavior of the Bergman metric},
preprint, 2015, arXiv:1504.02950.

\bibitem{DFH:Bergman}
K. Diederich, J.E. Forn\ae ss, G. Herbort;
\emph{Boundary behavior of the Bergman metric},
in ``Complex Analysis of Several Variables" (Madison, Wis., 1982), 59-67,
Proc. Sympos. Pure Math. {\bf 41}, Amer. Math. Soc., Providence, RI, 1984.

\bibitem{DFW:squeezing}
K. Diederich, J.E. Forn\ae ss, E.F. Wold;
\emph{A characterization of the ball},
Internat. J. Math. {\bf 27} (2016), 1650078, 5 pp.

\bibitem{Forelli}
F. Forelli;
\emph{Pluriharmonicity in terms of harmonic slices},
Math. Scand. {\bf 41} (1977), 358-364.

\bibitem{FL:metrics}
J.E. Forn\ae ss, L. Lee;
\emph{Kobayashi, Carath\'{e}odory and Sibony metric},
Complex Var. Elliptic Equ. {\bf 54} (2009), 293-301.

\bibitem{FS:squeezing}
J.E. Forn\ae ss, N. Shcherbina;
\emph{A domain with non-plurisubharmonic squeezing function},
J. Geom. Anal., to appear.

\bibitem{FW:squeezing}
J.E. Forn\ae ss, E.F. Wold;
\emph{An estimate for the squeezing function and estimates of invariant metrics},
in ``Complex Analysis and Geometry", 135-147, Springer Proc. Math. Stat. {\bf 144}, Springer, Tokyo, 2015.

\bibitem{Frankel}
S. Frankel;
\emph{Applications of affine geometry to geometric function theory in several complex variables. I. Convergent rescalings and intrinsic quasi-isometric structure},
in ``Several Complex Variables and Complex Geometry", Part 2 (Santa Cruz, CA, 1989), 183-208, Proc. Sympos. Pure Math., vol. {\bf 52}, Part 2, Amer. Math. Soc., Providence, RI, 1991.

\bibitem{KZ:squeezing}
K.-T. Kim, L. Zhang;
\emph{On the uniform squeezing property of bounded convex domains in $\cv^n$},
Pacific J. Math. {\bf 282} (2016), 341-358.

\bibitem{K:indicatrix}
S. Kobayashi;
\emph{Intrinsic distances, measures and geometric function theory},
Bull. Amer. Math. Soc. {\bf 82} (1976), 357-416.

\bibitem{LSY1}
K. Liu, X. Sun, S.-T. Yau;
\emph{Canonical metrics on the moduli space of Riemann surfaces. I},
J. Differential Geom. {\bf 68} (2004), 571-637.

\bibitem{LSY2}
K. Liu, X. Sun, S.-T. Yau;
\emph{Canonical metrics on the moduli space of Riemann surfaces. II},
J. Differential Geom. {\bf 69} (2005), 163-216.

\bibitem{Y:squeezing}
S.-K. Yeung;
\emph{Geometry of domains with the uniform squeezing property},
Adv. Math. {\bf 221} (2009), 547-569.

\end{thebibliography}
\end{document}